\documentclass[conference]{IEEEtran}
\IEEEoverridecommandlockouts
\usepackage{cite}
\usepackage{amsthm,amsmath,amssymb}
\usepackage{graphicx}
\newtheorem{theorem}{Theorem}
\newtheorem{corollary}{Corollary}
\newtheorem{statement}{Statement}

\newcommand{\conv}{\mathop{\rm conv}\nolimits}
\newcommand{\Vol}{\mathop{\rm Vol}\nolimits}
\newcommand{\RR}{{\mathbb R}}
\newcommand{\NN}{{\mathbb N}}

\newcommand{\PP}{{\mathsf P}}

\begin{document}

\title{Linear and Fisher Separability
of Random Points in~the~$d$-dimensional Spherical Layer
\thanks{The work is supported by the Ministry of Education and Science of Russian Federation 
(project 14.Y26.31.0022).}}

\author{\IEEEauthorblockN{S.\,V.\,Sidorov \qquad N.\,Yu.\,Zolotykh}
\IEEEauthorblockA{\textit{Institute of Information Technologies, Mathematics and Mechanics} \\
\textit{Lobachevsky State University of Nizhni Novgorod}\\
Nizhni Novgorod, Russia \\
Email: sergey.sidorov@itmm.unn.ru,
nikolai.zolotykh@itmm.unn.ru}}


\maketitle

\begin{abstract}
Stochastic separation theorems play important role in high-dimensional data analysis and machine learning. It turns out that in high dimension any point of a random set of points can be separated from other points by a hyperplane with high probability even if the number of points is exponential in terms of dimension. This and similar facts can be used for constructing correctors for artificial intelligent systems, for determining an intrinsic dimension of data and for explaining various natural intelligence phenomena.
In this paper, we refine the estimations for the number of points and for the probability in stochastic separation theorems, thereby strengthening some results obtained earlier. We propose the boundaries for linear and Fisher separability, when the points are drawn randomly, independently and uniformly from a $d$-dimensional spherical layer.
These results allow us to better outline the applicability limits of the stochastic separation theorems in applications.
\end{abstract}


\section{Introduction}

Recently, stochastic separation theorems \cite{GorbanTyukin2017} have been widely used in machine learning for constructing correctors and ensembles of correctors of artificial intelligence systems \cite{GorbanGrechukTyukin2018,GorbanGolubkovGrechukMirkesTyukin2018}, for determining the intrinsic dimension of data sets \cite{AlberganteBacZinovyev2019} and for explaining various natural intelligence phenomena, such as grandmother's neuron \cite{GorbanMakarovTyukin2019} etc.

If the dimension of the data is high, then any sample of the data set can be separated from all other samples by a hyperplane (or even Fisher discriminant -- as a special case) with a probability close to $1$ even the number of samples is exponential in terms of dimension. So, high-dimensional datasets exhibit fairly simple geometric properties. Due to the applications mentioned above the theorems of such kind can be considered as a manifestation of so called the {\em blessing of dimensionality} phenomenon \cite{GorbanTyukin2017,GorbanMakarovTyukin2020}.

In its usual form a stochastic separation theorem is formulated as follows.
A random $n$-element set in $\RR^d$
is linearly separable with probability $p>1-\vartheta$, if $n < a e^{b d}$. The exact form of the exponential function depends on the probability distribution that determines how the random set is drawn, and on the constant $\vartheta$ ($0<\vartheta<1$). In particular, uniform distributions with different support are considered in \cite{GorbanTyukin2017,GorbanBurtonTyukin2019,SidorovZolotykh2019}. Wider classes of distributions (including non-i.i.d.) are considered in \cite{GorbanGrechukTyukin2018}. Roughly speaking, these classes consist of distributions without sharp peaks in sets with exponentially small volume. Estimates for product distributions in the cube and the standard normal distribution is obtained in \cite{Grechuk2019}.

We note that there are many algorithms for constructing a functional separating a point from all other points in a data set
(Fisher linear discriminant, linear programming algorithm, support vector machine, Rosenblatt perceptron etc.). 
Among all these methods the computationally cheapest is Fisher discriminant \cite{GorbanGolubkovGrechukMirkesTyukin2018}. Other advantages of the Fisher discriminant are its simplicity and the robustness.

The papers \cite{GorbanTyukin2017,GorbanBurtonTyukin2019,GorbanGrechukTyukin2018,GorbanGolubkovGrechukMirkesTyukin2018}
deal with only Fisher separability, whereas \cite{SidorovZolotykh2019} considered a (more general) linear separability. A comparison of the estimations for linear and Fisher separability allows us to clarify the applicability boundary of these methods, namely, to answer the question, for what $d$ and $n$ it suffices to use only Fisher separability and there is no need to search a more sophisticated linear discriminant.

In \cite{SidorovZolotykh2019} there were obtained estimations for the cardinality of the set of points that guarantee its linear separability when the points are drawn randomly, independently and uniformly from a $d$-dimensional spherical layer and from the unit cube. These results give more accurate estimates than the bounds obtained in \cite{GorbanTyukin2017,GorbanBurtonTyukin2019} for Fisher separability.
Here we give even more precise estimations for the number of points in the spherical layer to guarantee their linear separability. Also, we report the results of computational experiments comparing the theoretical estimations for the probability of the linear and Fisher separabilities with the corresponding experimental frequencies and discuss them.

\section{Definitions}


A point $X\in \RR^d$ is {\em linearly separable} from the set $M\subset\RR^d$ if there exists a hyperplane separated $X$ from $M$, i.e. 
there exists $A_X\in\RR^d$ such that $(A_X,X) > (A_X,Y)$ 
for all $Y\in M$.

A point $X\in \RR^d$ is {\em Fisher separable} from the set $M\subset\RR^d$ if the inequality $(X,Y) < (X,X)$ holds for all $Y\in M$ \cite{GorbanGolubkovGrechukMirkesTyukin2018,GorbanGrechukTyukin2018}.

A set of points $\{X_1,\ldots,X_n\}\subset\RR^d$ is called {\em $1$-convex} \cite{BaranyFuredi1988} 
or {\em linearly separable} \cite{GorbanTyukin2017}
if any point $X_i$ is linearly separable from all other points in the set, or, in other words, the set of vertices of their convex hull, $\conv(X_1,\ldots,X_n)$, coincides with $\{X_1,\ldots,X_n\}$.

The set $\{X_1,\ldots,X_n\}$ is called {\em Fisher separable} if $(X_i,X_j) < (X_i,X_i)$ for all $i$, $j$, such that $i\ne j$ 
\cite{GorbanGolubkovGrechukMirkesTyukin2018,GorbanGrechukTyukin2018}.

Fisher separability implies linear separability but not vice versa
(even if the set is centered and normalized to unit variance). 
Thus, if $M\subset\RR^d$ is a random set of points from a certain probability distribution, 
then the probability that $M$ is linearly separable is not less than 
the probability that $M$ is Fisher separable.





Let $B_d=\{X\in\RR^d:~\|X\|\leq 1\}$ be the $d$-dimensional unit ball centered at the origin
($\|X\|$ means Euclidean norm), $rB_d$ is the $d$-dimensional ball of radius $r<1$ centered at the origin. 

Let $M_n=\{X_1,\ldots,X_n\}$ be the set of points chosen randomly, independently, according to the uniform distribution on the spherical layer $B_d\setminus rB_d$. 
Denote by 
$P(d, r, n)$ the probability that $M_n$ is linearly separable, and by $P^{F}(d, r, n)$ the probability that $M_n$ is Fisher separable.

Denote by 
$P_1(d, r, n)$ the probability that a random point chosen according to the uniform distribution on  $B_d\setminus rB_d$ is separable from $M_n$, and by
$P_1^F(d, r, n)$ the probability that a random point is Fisher separable from $M_n$.

\section{Previous works}

In \cite{GorbanTyukin2017} it was shown (among other results) that for all $r$, $\vartheta$, $n$, $d$, where $0<r<1$, $0<\vartheta<1$, $d\in \NN$, if
\begin{equation}\label{eq1}
n<\left(\frac{r}{\sqrt{1-r^2}}\right)^d\left(\sqrt{1+\frac{2\vartheta(1-r^2)^{d/2}}{r^{2d}}}-1\right),
\end{equation}
then $M_n$ is Fisher separable with a probability greater than $1-\vartheta$, i.e. $P^F(d,r,n) > 1-\vartheta$.

The following statements are proved in \cite{GorbanBurtonTyukin2019}.

\begin{itemize}
	\item  For all $r$, where $0<r<1$, and for any $d\in\NN$
\begin{equation}\label{f-PF1}
P_1^F(d,r,n) > (1-r^d)\left(1-\frac{(1-r^2)^{d/2}}{2}\right)^n.
\end{equation}

	\item For all $r$, $\vartheta$, where $0<r<1$, $0<\vartheta<1$, and for sufficiently large $d$, if 
\begin{equation}
\label{f-nF1}
n<\frac{\vartheta}{(1-r^2)^{d/2}}
\end{equation}
then $P^F_1(d,r,n) > 1-\vartheta$.

	\item For all $r$, where $0<r<1$, and for any $d\in\NN$ 
\begin{equation}\label{f-PF}
P^F(d,r,n) >
\left[(1-r^d)\left(1-(n-1)\frac{(1-r^2)^{d/2}}{2}\right)\right]^n.
\end{equation}
	\item For all $r$, $\vartheta$, where $0<r<1$, $0<\vartheta<1$ and for sufficiently large $d$, if 
\begin{equation}\label{f-nF}
n<\frac{\sqrt{\vartheta}}{(1-r^2)^{d/4}}
\end{equation}
then $P^F(d,r,n)>1-\vartheta$.
\end{itemize}

The authors of \cite{GorbanTyukin2017},\cite{GorbanBurtonTyukin2019} formulate their results for linearly separable sets of points, but in fact
in the proofs they used that the sets are only Fisher separable.

Note that all estimates (\ref{eq1})--(\ref{f-nF}) require $0< r <1$ with strong inequality. This means that they are inapplicable for (maybe the most interesting) case $r = 0$.

The both estimates $(\ref{eq1})$, $(\ref{f-nF})$ are exponentially dependent on $d$ for fixed $r$, $\vartheta$ 
and the estimate $(\ref{eq1})$ is weaker than $(\ref{f-nF})$
(see Section \ref{sec-ComparisonOfResults}).


In \cite{SidorovZolotykh2019} it was proved that
if $0\leq r<1,$ $0<\vartheta<1,$  
$$
n<\sqrt{\vartheta 2^{d}(1-r^d)},
$$ 
then $P(d,r,n)>1-\vartheta.$
Here we improve this bound (see Corollary~\ref{th-n}) and also give the estimates for $P_1(d,r,n)$ and $P(d,r,n)$ and compare them with known estimates (\ref{f-PF1}), (\ref{f-PF}) for $P^F_1(d,r,n)$ and $P^F(d,r,n)$.



\section{New results}


The following theorem gives a probability of the linear separability 
of a random point from a random $n$-element set $M_n=\{X_1,\ldots,X_n\}$ in $B_d\setminus rB_d.$ 
The proof uses an approach borrowed from \cite{BaranyFuredi1988}, \cite{Elekes1986}.

\begin{theorem}\label{th-P1}
Let $0\leq r<1$, $d\in\NN$. Then  
\begin{equation}\label{f-P1}
P_1(d,r,n)>1-\frac{n}{2^{d}}.  
\end{equation}
\end{theorem}

\begin{proof} A random point $Y$ is linearly separable from $M_n=\{X_1,\ldots,X_n\}$ if and only if $Y\notin \conv(M_n).$ Denote this event by $C.$ 
Thus $P_1(d,r,n)=\PP(C).$
Let us find the upper bound for the probability of the event $\overline{C}.$ 
This event means that the point $Y$ belongs to the convex hull of $M_n.$ 
Since the points in $ M_n $ have the uniform distribution, then the probability of $\overline{C}$ is 
$$
P(\overline{C})=\frac{ \Vol\Bigl(\conv(M_n)\setminus \Bigl(\conv(M_n)\cap rB_d\Bigl)\Bigr)}{\Vol(B_d)-\Vol(rB_d)}.
$$

\begin{figure*}[t!]%
	\centering
	\includegraphics[width=1\textwidth]{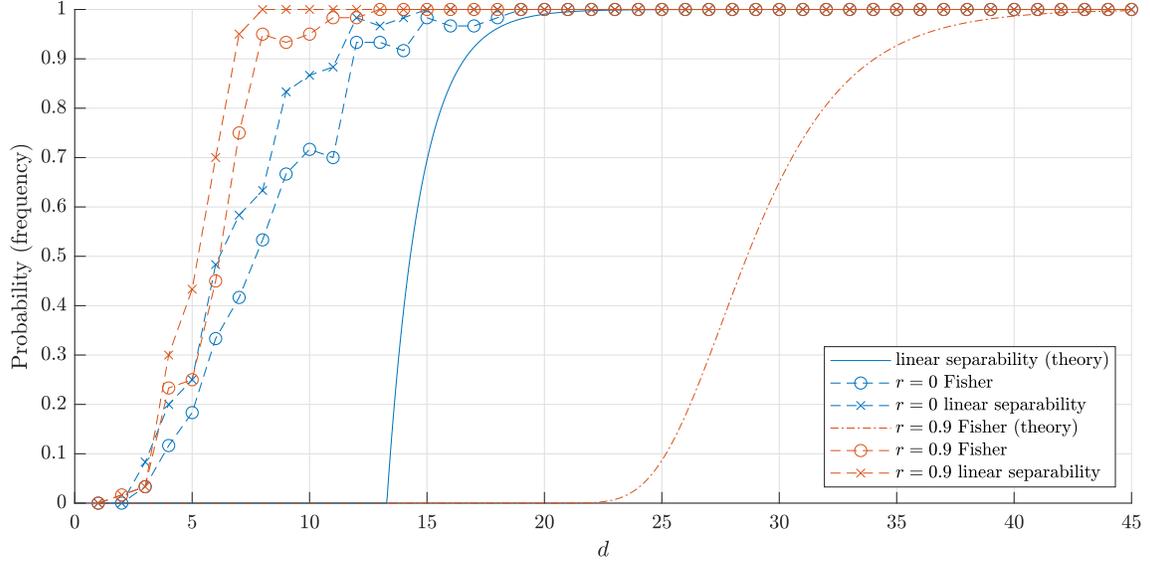}
	\caption{The probability (frequency) that a random point is linearly (or Fisher) separable from a set of $n=10000$ random points in the layer $B_d\setminus r B_d$.
	The blue solid line corresponds to the theoretical bound (\ref{f-P1}) for the linear separability.
	The red dash-dotted line represents the theoretical bound (\ref{f-PF1}) for the Fisher separability.
	The crosses (circles) correspond to the empirical frequencies for linear (and, respectively, Fisher) separability obtained in $60$ trials for each dimension $d$.}
	\label{fig_n1_n10000}
\end{figure*}

Let us estimate the numerator of this fraction. We denote by $S_i$ the ball with center at the origin and with the diameter $OX_i.$  We denote by $T_i$ the ball with center at the origin and with the diameter $r$ inside the ball $S_i.$ Then 

$$\conv(M_n)\setminus (\conv(M_n)\cap rB_d)\subseteq \bigcup\limits_{i=1}^{n}(S_i\setminus T_i)$$ 
and 
$$\Vol\Bigl(\conv(M_n)\setminus \Bigl(\conv(M_n)\cap rB_d\Bigl)\Bigr)\leq \sum\limits_{i=1}^{n}\Vol(S_i\setminus T_i)=$$
$$=\sum\limits_{i=1}^{n}(\Vol(S_i) - \Vol(T_i))=
\sum\limits_{i=1}^{n}\left(\Vol(S_i) - \gamma_d\left(\frac{r}{2}\right)^d\right)\leq$$
$$\leq\sum\limits_{i=1}^{n}\left(\gamma_d \left(\frac{1}{2}\right)^d - \gamma_d \left(\frac{r}{2}\right)^d\right)=\frac{n\gamma_d (1-r^d)}{2^d},$$ 
where $\gamma_d$ is the volume of a ball of radius $1$.

Hence 
$$
\PP(\overline{C})\leq\frac{\frac{n\gamma_d (1-r^d)}{2^d}}{\gamma_d (1-r^d)}=\frac{n}{2^d}
$$ 
and
$$
\PP(C) = 1-\PP(\overline{C})\geq 1-\frac{n}{2^d}.
$$ 
\end{proof}

Note that the bound (\ref{f-P1}) obtained in Theorem \ref{th-P1} doesn't depend on $r$. Nevertheless the bound is quite accurate (in the sense that  \ref{th-P1} shows behaviour close to empirical values.) as is illustrated with Figure \ref{fig_n1_n10000}.
The results of the experiment show that the probabilities $P_1(d,r,n)$ and $P_1^F(d,r,n)$ are quite close and the theoretical bound (\ref{f-P1}) compared with (\ref{f-PF1}) approximates well the both probabilities. 

It is clear that the probabilities must increase monotonously when $d$ increases, but in the real experiment the frequency can not coincide precisely with the probability and it can have non-monotonic behaviour. In our experiment (with $60$ trials for each $d$) it is non-monotonous.

The following corollary gives an improved estimate for the number of points $n$ guaranteeing 
the linear separability of a random point from a random $n$-element set $M_n$ in $B_d\setminus rB_d$ 
with probability at least $1-\vartheta.$

\begin{corollary}
Let $0\leq r<1,$ $0<\vartheta<1,$  
\begin{equation}\label{f-n1}
n<\vartheta 2^{d}.  
\end{equation} 
Then $P_1(d,r,n)>1-\vartheta.$
\end{corollary}

\begin{proof} 
If $n$ satisfies the condition $n<\vartheta 2^{d}$, then the inequality $P_1(d,r,n)>1-\vartheta$ holds by the previous theorem. 
\end{proof}

\begin{figure*}[t!]%
	\centering
	\includegraphics[width=1\textwidth]{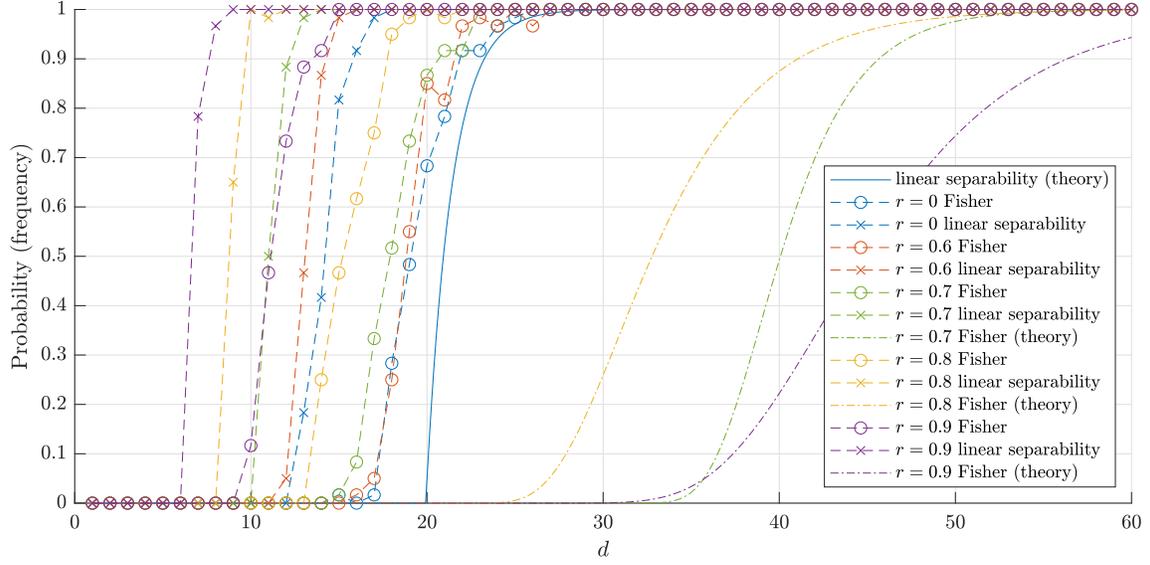}
	\caption{The probability (frequency) that the set of $n=1000$ random points in the layer $B_d\setminus r B_d$ is linearly or Fisher separable.
	The blue solid line corresponds to the theoretical bound (\ref{f-P}) for the linear separability obtained in Theorem~\ref{th-P}.
	The dash-dotted lines represent the theoretical bound (\ref{f-PF}) for the Fisher separability.
	The crosses (circles) correspond to the empirical frequencies for linear (and, respectively, Fisher) separability obtained in $60$ trials for each dimension $d$.}
	\label{fig_n1000}
\end{figure*}

\begin{figure*}[h!]%
	\centering
	\includegraphics[width=1\textwidth]{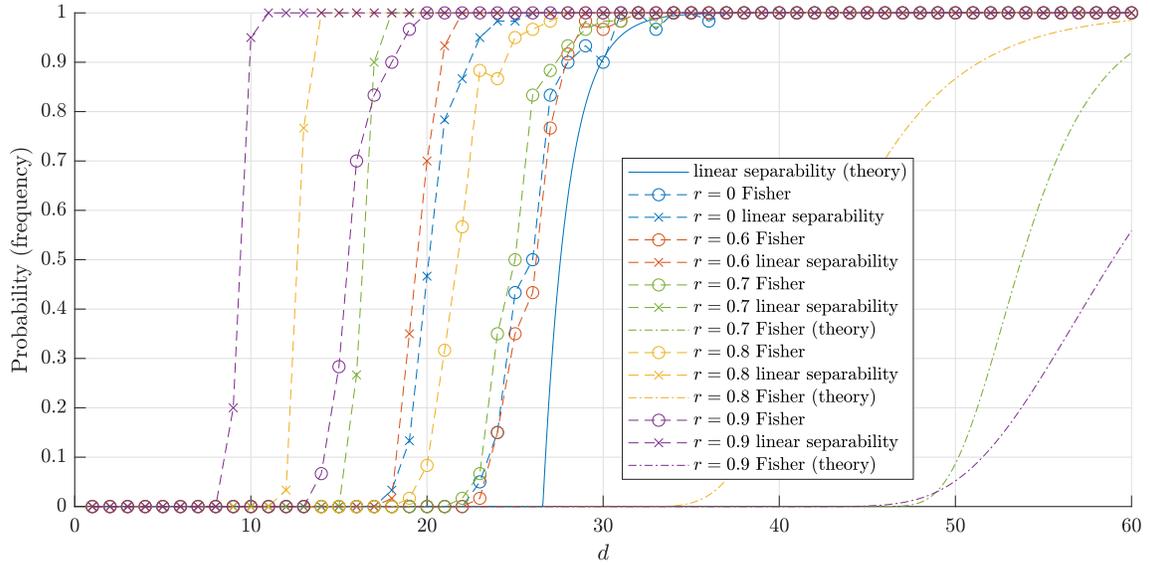}
	\caption{The probability (frequency) that the set of $n=10000$ random points in the layer $B_d\setminus r B_d$ is linearly or Fisher separable.
	The notations are the same as on Figure~\ref{fig_n1000}}
	\label{fig_n10000}
\end{figure*}

The following theorem gives the probability of the linear separability 
of a random $n$-element set $M_n$ in $B_d\setminus rB_d.$

\begin{theorem}\label{th-P}
Let $0\leq r<1$, $d\in\NN$. Then 
\begin{equation}\label{f-P}
P(d,r,n)>1-\frac{n(n-1)}{2^{d}}. 
\end{equation}
\end{theorem}

\begin{proof} Denote by $A_n$ the event that $M_n$ is linearly separable and 
denote by $C_i$ the event that $X_i\notin \conv(M_n\setminus\{X_i\})$ ($i=1,\ldots,n$). Thus $P(d,r,n)=\PP(A_n).$
Clearly $A_n=C_1\cap\ldots\cap C_n$ and $\PP(A_n)=\PP(C_1\cap\ldots\cap C_n)=
1-\PP(\overline{C}_1\cup\ldots\cup \overline{C}_n)\geq 1-\sum\limits_{i=1}^{n}\PP(\overline{C}_i).$ 
Let us find the upper bound for the probability of the event $\overline{C}_i.$ 
This event means that the point $X_i$ belongs to the convex hull of the remaining points, i.e. $X_i\in \conv(M_n\setminus\{X_i\}).$ 
In the proof of the previous theorem, it was shown that 
$$
\PP(\overline{C}_i)\leq\frac{n-1}{2^d} \qquad (i=1,\ldots,n).
$$ 
Hence 
$$
\PP(A_n)\geq 1-\sum\limits_{i=1}^{n}\PP(\overline{C}_i)\geq 1-\frac{n(n-1)}{2^d}.
$$ 
\end{proof}

Note that the bound (\ref{f-P}) obtained in Theorem \ref{th-P} doesn't depend on $r$, although $P(d,r,n)$ seems to increase monotonically with increasing $r$ (for a big enough $n$). Nevertheless the bound is quite accurate as is illustrated with Figures \ref{fig_n1000}, \ref{fig_n10000}.
The results of the experiment show that the probabilities $P_1(d,r,n)$ and $P_1^F(d,r,n)$ are quite close and the theoretical bound (\ref{f-P}) compared with (\ref{f-PF}) approximates well the both probabilities. 

Another important conclusion from the experiment is as follows. Despite the fact that both probabilities $P_F(d,r,n)$ $P(d,r,n)$ are close to $1$ for sufficiently big $d$, the ``threshold values'' for such a sufficiently big $d$ differ greatly. In other words, the blessing of dimensionality when using linear discriminants comes noticeably earlier than if we only use Fisher discriminants. This is achieved at the cost of constructing the usual linear discriminant in comparison with the Fisher one.

The following corollary gives an improved estimate for the number of points $n$ guaranteeing 
the linear separability of a random $n$-element set $M_n$ in $B_d\setminus rB_d$ 
with probability at least $1-\vartheta.$ 
This result strengthens the result obtained in \cite{SidorovZolotykh2019}.

\begin{corollary}\label{th-n}
Let $0\leq r<1,$ $0<\vartheta<1,$  
\begin{equation}\label{f-n}
n<\sqrt{\vartheta 2^{d}}.  
\end{equation}
Then $P(d,r,n)>1-\vartheta.$
\end{corollary}

\begin{proof} 
If $n$ satisfies the condition $n<\sqrt{\vartheta 2^{d}}$, then by the previous theorem 
$$
P(d,r,n)>1-\frac{n(n-1)}{2^{d}}>1-\frac{n^2}{2^{d}}>1-\vartheta.$$  
\end{proof}

\section{Comparison of the results}\label{sec-ComparisonOfResults}

The following statement establishes the asymptotics of the bound $(\ref{eq1})$.

\begin{statement} 
Let $g=\left(\frac{r}{\sqrt{1-r^2}}\right)^d\left(\sqrt{1+\frac{2\vartheta(1-r^2)^{d/2}}{r^{2d}}}-1\right),$ $0<r<1,$ $0<\vartheta<1,$ $d\in\NN$. If $r$ and $\vartheta$ are fixed then the following asymptotic estimates hold:

\begin{enumerate}
	\item[\rm 1.] $g\sim\frac{\vartheta}{r^d},$ if  $\sqrt{\frac{\sqrt{5}-1}{2}}<r<1.$

	\item[\rm 2.] $g=\frac{2\vartheta}{r^d(\sqrt{1+2\vartheta}+1)}=
\frac{\sqrt{1+2\vartheta}-1}{r^d}=(\sqrt{1+2\vartheta}-1)(\frac{\sqrt{5}+1}{2})^{d/2},$ if $r=\sqrt{\frac{\sqrt{5}-1}{2}}.$

	\item[\rm 3.] $g\sim\frac{\sqrt{2\vartheta}}{(1-r^2)^{d/4}},$ if $0<r<\sqrt{\frac{\sqrt{5}-1}{2}}.$
\end{enumerate}
\end{statement}

\begin{proof}
We have $$g=\frac{\left(\frac{r}{\sqrt{1-r^2}}\right)^d\frac{2\vartheta(1-r^2)^{d/2}}{r^{2d}}}{\sqrt{1+\frac{2\vartheta(1-r^2)^{d/2}}{r^{2d}}}+1}=
\frac{2\vartheta}{r^d\left(\sqrt{1+2\vartheta\left(\frac{\sqrt{1-r^2}}{r^{2}}\right)^d}+1\right)}.$$ 

If $0<\frac{\sqrt{1-r^2}}{r^{2}}<1$ then $g\sim\frac{\vartheta}{r^d}$.

If $\frac{\sqrt{1-r^2}}{r^{2}}=1$ then $g=\frac{2\vartheta}{r^d(\sqrt{1+2\vartheta}+1)}=
\frac{\sqrt{1+2\vartheta}-1}{r^d}$.

If $\frac{\sqrt{1-r^2}}{r^{2}}>1$ then $g\sim\frac{2\vartheta}{r^d\sqrt{2\vartheta\left(\frac{\sqrt{1-r^2}}{r^{2}}\right)^d}}=\frac{\sqrt{2\vartheta}}{(1-r^2)^{d/4}}$.

The equality $\frac{\sqrt{1-r^2}}{r^{2}}=1$ holds if $r^4+r^2-1=0,$ that is $r^2=\frac{\sqrt{5}-1}{2},$ $r=\sqrt{\frac{\sqrt{5}-1}{2}}.$ The inequality $0<\frac{\sqrt{1-r^2}}{r^{2}}<1$ holds if  $r^4+r^2-1>0,$ that is for $\sqrt{\frac{\sqrt{5}-1}{2}}<r<1.$ The inequality $\frac{\sqrt{1-r^2}}{r^{2}}>1$ holds if   $r^4+r^2-1<0,$ that is for $0<r<\sqrt{\frac{\sqrt{5}-1}{2}}.$
\end{proof}

Let us compare the bound $(\ref{f-nF})$ with the bound $(\ref{eq1})$
proposed in \cite{GorbanTyukin2017}, \cite{GorbanBurtonTyukin2019}.

\begin{corollary} 
Let $f=\frac{\sqrt{\vartheta}}{(1-r^2)^{d/4}},$ $g=\left(\frac{r}{\sqrt{1-r^2}}\right)^d\left(\sqrt{1+\frac{2\vartheta(1-r^2)^{d/2}}{r^{2d}}}-1\right),$ $0<r<1,$ $0<\vartheta<1,$ $d\in\NN.$ If $r$ and  
$\vartheta$ are fixed then the following asymptotic estimates of the quotient $\frac{f}{g}$ hold:

\begin{enumerate}
	\item[\rm 1.] $\frac{f}{g}\sim\frac{1}{\sqrt{\vartheta}}\left(\frac{r^2}{\sqrt{1-r^2}}\right)^{d/2}\rightarrow\infty,$ if  $\sqrt{\frac{\sqrt{5}-1}{2}}<r<1$.
	\item[\rm 2.] $\frac{f}{g}=\frac{\sqrt{1+2\vartheta}+1}{2\sqrt{\vartheta}}>1,$ if $r=\sqrt{\frac{\sqrt{5}-1}{2}}$.
	\item[\rm 3.] $\frac{f}{g}\sim\frac{1}{\sqrt{2}},$ if $0<r<\sqrt{\frac{\sqrt{5}-1}{2}}$.
\end{enumerate}
\end{corollary}

\begin{proof} 
If $\sqrt{\frac{\sqrt{5}-1}{2}}<r<1,$ then $\frac{f}{g}\sim \frac{\sqrt{\vartheta}}{(1-r^2)^{d/4}}\frac{r^d}{\vartheta}=
\frac{1}{\sqrt{\vartheta}}\left(\frac{r^2}{\sqrt{1-r^2}}\right)^{d/2}\rightarrow\infty$ for $d\rightarrow\infty,$ since $\frac{r^2}{\sqrt{1-r^2}}>1$ for $r>\sqrt{\frac{\sqrt{5}-1}{2}}.$

If $r=\sqrt{\frac{\sqrt{5}-1}{2}},$ then $\frac{f}{g}= \frac{\sqrt{\vartheta}}{(1-r^2)^{d/4}}\frac{r^d(\sqrt{1+2\vartheta}+1)}{2\vartheta}=\frac{\sqrt{1+2\vartheta}+1}{2\sqrt{\vartheta}}\left(\frac{r^2}{\sqrt{1-r^2}}\right)^{d/2}=\frac{\sqrt{1+2\vartheta}+1}{2\sqrt{\vartheta}}>1.$ 

If $0<r<\sqrt{\frac{\sqrt{5}-1}{2}},$ then  
$\frac{f}{g}\sim \frac{\sqrt{\vartheta}}{(1-r^2)^{d/4}}\frac{(1-r^2)^{d/4}}{\sqrt{2\vartheta}}=\frac{1}{\sqrt{2}}.$ 
\end{proof}

The following statement compares estimates of the number of points that guarantee 
linear separability of a random points in the spherical layer obtained in \cite{GorbanBurtonTyukin2019} and in Corollary~\ref{th-n}.

\begin{statement}\label{Th_spherical_segment} 
Let $f=\sqrt{\vartheta 2^{d}},$ $g=\frac{\sqrt{\vartheta}}{(1-r^2)^{d/4}},$ $0<r<1,$ $0<\vartheta<1,$ $d\in\NN.$ If $r$ and  
$\vartheta$ are fixed then 
$$\frac{f}{g}\sim (2\sqrt{1-r^2})^{d/2}.$$ 

If $0<r<\frac{\sqrt{3}}{2}$ then $\frac{f}{g}\rightarrow\infty.$

If $r=\frac{\sqrt{3}}{2}$ then $\frac{f}{g}\rightarrow 1.$

If $\frac{\sqrt{3}}{2}< r <1$ then $\frac{f}{g}\rightarrow 0.$

\end{statement} 

\begin{proof} 
$\frac{f}{g} = \frac{\sqrt{\vartheta 2^d}(1-r^2)^{d/4}}{\sqrt{\vartheta}}=
(2\sqrt{1-r^2})^{d/2}$ for $0\leq r <1.$ For $0<r<\frac{\sqrt{3}}{2}$ inequality   $2\sqrt{1-r^2}>1$ holds so $\frac{f}{g}\rightarrow\infty.$
 For $\frac{\sqrt{3}}{2}< r <1$ inequality $2\sqrt{1-r^2}<1$ holds so $\frac{f}{g}\rightarrow 0.$ 
For $r=\frac{\sqrt{3}}{2}$ equality $2\sqrt{1-r^2}=1$ holds so $\frac{f}{g}\rightarrow 1.$

\end{proof}

Let us compare the rate of convergence to $1$ of estimates (\ref{f-PF}) and (\ref{f-P}).

\begin{statement}\label{Th_asymp_prob} 
Let  $g=1-\left[(1-r^d)\left(1-(n-1)\frac{(1-r^2)^{d/2}}{2}\right)\right]^n,$ $0<r<1,$ $d, n\in\NN.$ If $r$ and $n$ are fixed then the following asymptotic estimates hold:

\begin{enumerate}
	\item[\rm 1.] $g\sim n r^d,$ if  $\frac{\sqrt{2}}{2}<r<1.$

	\item[\rm 2.] $g\sim \frac{n(n-1)}{2}(1-r^2)^{d/2},$ if $0<r<\frac{\sqrt{2}}{2}.$
	
	\item[\rm 3.] $g\sim \frac{n(n+1)}{2}\frac{1}{2^{d/2}},$ if $r=\frac{\sqrt{2}}{2}.$
\end{enumerate}
\end{statement} 

\begin{proof} 

Let $y=\sqrt{1-r^2},$ $a=\frac{n-1}{2}.$ Then $g=1-(1-r^d)^n(1-ay^d)^n=1-(1-nr^d+\frac{n(n-1)}{2}x^{2d}-\ldots)(1-nay^d+\frac{n(n-1)}{2}a^2y^{2d}-\ldots)=1-(1-nr^d-nay^d+\frac{n(n-1)}{2}a^2y^{2d}+n^2ar^dy^d+\frac{n(n-1)}{2}r^{2d}+\ldots)=nr^d+nay^d-\frac{n(n-1)}{2}a^2y^{2d}-n^2ar^dy^d-\frac{n(n-1)}{2}r^{2d}+\ldots.$

If $\frac{\sqrt{2}}{2}<r<1,$ then $r>y$ so $g=nr^d\left(1+a(\frac{y}{r})^d-a^3y^d(\frac{y}{r})^d-nay^d-ar^d+\ldots\right)\sim nr^d.$

If $0<r<\frac{\sqrt{2}}{2},$ then $r<y$ so $g=ny^d\left((\frac{r}{y})^d+a-a^3y^d-nar^d-a(\frac{r}{y})^dr^d+\ldots\right)\sim nay^d.$

If $r=\frac{\sqrt{2}}{2},$ then $r=y$ so $g=nr^d(1+a)-na^3r^{2d}-n^2ar^{2d}-nar^{2d}+\ldots\sim nr^d(1+a)=\frac{n(n+1)}{2}r^d=\frac{n(n+1)}{2}\frac{1}{2^{d/2}}.$

\end{proof}

The following statement compares estimates (\ref{f-PF}) and (\ref{f-P}).

\begin{statement}\label{Th_prob_compare} 
Let $f=\frac{n(n-1)}{2^{d}},$ $g=1-\left[(1-r^d)\left(1-(n-1)\frac{(1-r^2)^{d/2}}{2}\right)\right]^n,$ $0<r<1,$ $d, n\in\NN.$ If $r$ and $n$ are fixed then 

\begin{enumerate}
	\item[\rm 1.] $\frac{g}{f}\sim \frac{(2r)^d}{n-1}\rightarrow\infty,$ 
	if  $\frac{\sqrt{2}}{2}<r<1.$

	\item[\rm 2.] $\frac{g}{f}\sim \frac{(4(1-r^2))^{d/2}}{2}\rightarrow\infty,$ if $0<r<\frac{\sqrt{2}}{2}.$
	
	\item[\rm 3.] $\frac{g}{f}\sim \frac{2^{d/2}(n+1)}{2(n-1)}\rightarrow\infty,$ if $r=\frac{\sqrt{2}}{2}.$

\end{enumerate}

\end{statement} 

\begin{proof} 

If $\frac{\sqrt{2}}{2}<r<1,$ then  $g\sim nr^d$ so  $\frac{g}{f}\sim \frac{nr^d}{\frac{n(n-1)}{2^{d}}}= \frac{(2r)^d}{n-1}\rightarrow\infty$ as $2r>\sqrt{2}.$

If $0<r<\frac{\sqrt{2}}{2},$ then  $g\sim \frac{n(n-1)}{2}(1-r^2)^{d/2}$ so 
$\frac{g}{f}\sim \frac{\frac{n(n-1)}{2}(1-r^2)^{d/2}}{\frac{n(n-1)}{2^{d}}}=\frac{(4(1-r^2))^{d/2}}{2} \rightarrow\infty$ as $4(1-r^2)>2.$

If $r=\frac{\sqrt{2}}{2},$ then $g\sim\frac{n(n+1)}{2}\frac{1}{2^{d/2}}$ so 
$\frac{g}{f}\sim \frac{\frac{n(n+1)}{2}\frac{1}{2^{d/2}}}{\frac{n(n-1)}{2^{d}}} = \frac{2^{d/2}(n+1)}{2(n-1)}\rightarrow\infty.$ 

\end{proof}

Thus,  the estimate (\ref{f-P}) tends to $1$ faster than the estimate (\ref{f-PF}) for all $0<r<1.$

\medskip

\section{Conclusion}

In this paper we refined the bounds for the number of points and for the probability in stochastic separation theorems. We gave new bounds for linear separability, when the points are drawn randomly, independently and uniformly from a $d$-dimensional spherical layer or ball.
These results refine some results obtained in \cite{GorbanTyukin2017,GorbanBurtonTyukin2019,SidorovZolotykh2019} and allow us to better understand the  applicability limits of the stochastic separation theorems for high-dimensional data mining and machine learning problems.

One of the main results of the experiment comparing linear and Fisher separabilities is as follows. The blessing of dimensionality when using linear discriminants can come noticeably earlier (for smaller values of $d$) than if we only use \newpage 
\parindent=0pt Fisher discriminants. 
This is achieved at the cost of constructing the usual linear discriminant in comparison with the Fisher one.

\section*{Acknowledgements}

Authors are grateful to A.\,N.\,Gorban for useful discussions.

\end{document}